\pdfoutput=1
\documentclass[11pt]{amsart}
\usepackage{tabularx,booktabs,tikz}
\usepackage{caption}
\usepackage{amsmath}
\usepackage{amsfonts}
\DeclareMathOperator{\arcsinh}{arcsinh}
\usepackage{amscd}
\usepackage{amsthm}
\usepackage{amssymb} \usepackage{latexsym}
\usepackage{mathrsfs}
\usepackage{eufrak}
\usepackage{euscript}
\usepackage{epsfig}
\usepackage{graphics}
\usepackage{array}
\usepackage{enumerate}
\usepackage{dsfont}
\usepackage{color}
\usepackage{wasysym}
\usepackage{hyperref}
\usepackage{pdfsync}

\newcommand{\bel}[1]{\begin{equation}\label{#1}}

\newcommand{\be}{\begin{equation}}

\newcommand{\ba}{\begin{eqnarray}}
\newcommand{\ea}{\end{eqnarray}}

\newcommand{\qe}{\end{equation}}

\newcommand{\Hmm}[1]{\leavevmode{\marginpar{\tiny%
$\hbox to 0mm{\hspace*{-0.5mm}$\leftarrow$\hss}%
\vcenter{\vrule depth 0.1mm height 0.1mm width \the\marginparwidth}%
\hbox to
0mm{\hss$\rightarrow$\hspace*{-0.5mm}}$\\\relax\raggedright #1}}}

\newtheorem{theorem}{Theorem}[section]

\newtheorem{lemma}[theorem]{Lemma}
\newtheorem{corollary}[theorem]{Corollary}
\newtheorem{definition}[theorem]{Definition}

\newtheorem{remark}[theorem]{Remark}

\newtheorem{prop}[theorem]{Proposition}

\newcommand{\tm}{\begin{theorem}}
\newcommand{\tmd}{\end{theorem}}
\newcommand{\co}{\begin{corollary}}
\newcommand{\cod}{\end{corollary}}
\newcommand{\prp}{\begin{prop}}
\newcommand{\prpd}{\end{prop}}
\begin{document}

\title[Eigenvalue estimates for the fractional Laplacian]{Eigenvalue estimates for the fractional Laplacian on lattice subgraphs}

\author{Jiaxuan Wang}
\address{Jiaxuan Wang: School of Mathematical Sciences, Fudan University, Shanghai 200433, China}
\email{jiaxuanwang21@m.fudan.edu.cn}
%\address{Yohji Akama: Mathematical Institute, Graduate School of Science, Tohoku University,
%Sendai, 980-0845, Japan}
%\email{akama@math.tohoku.ac.jp}

%\author{Yanhui Su}
%\email{suyh@fzu.edu.cn}
%\address{Yanhui Su: College of Mathematics and Computer Science, Fuzhou University, Fuzhou 350116, China}

\begin{abstract}
%On the number of vertices with positive curvature of a planar graph with nonnegative combinatorial curvature

%\blue{Dedicate to ????}
% discrete harmonic functions on infinite penny graphs. For an infinite penny graph with bounded facial degree, we prove that the volume doubling property and the Poincar\'e inequality hold, which yields the Harnack inequality for positive harmonic functions. Moreover, we prove that the space of polynomial growth harmonic functions, or ancient solutions of the heat equation, with bounded growth rate has finite dimensional property.

We introduce the the fractional Laplacian on a subgraph of a graph with Dirichlet boundary condition. For a lattice graph, we prove the upper and lower estimates for the sum of the first $k$ Dirichlet eigenvalues of the fractional Laplacian, extending the classical results by Li-Yau and Kr\"{o}ger.

\end{abstract}
\maketitle

%\tableofcontents
%\tableofcontents

%\author{\large  Bobo Hua$^\ast$\ \ \ \  Yanhui Su$^{\dag}$} % ???????

\par
\maketitle

\bigskip

%BB: \begin{enumerate}%\item We denote by the distance function by $d$, not by $d^G.$
%%\item Euler characteristic of open surfaces, how to define???
%\item Check citations, number of theorems!
%%\item Let be, be.
%%\item Without loss of generality, we may always consider planar graphs with nonnegative curvature.
%%\item $\partial G$ to $\partial S.$
%\end{enumerate}

\section{Introduction}
The fractional Laplacian has significant applications across various research fields in mathematics. The properties and applications of this operator in $\mathbb{R}^n$ are increasingly discussed in literature, e.g. \cite{2016Eigenvalue,2015FRACTIONAL,2015Ten,2016The,2015The,2010Heat}. %\cite{2022Optimal}, \cite{2022Spectral}, \cite{Benzi2020}, \cite{21Decay}.For the discrete case, the literature about this operator tends to consider the fractional Laplacian on a whole graph.
In this paper, we introduce the fractional Laplacian on a locally finite connected graph and define the fractional Laplacian operator with Dirichlet boundary condition on a finite subgraph. For an integer lattice graph, we give the well-defined fractional Laplacian, and derive the upper and lower estimates for the Dirichlet eigenvalues of the fractional Laplacian.

In a bounded domain $\Omega \subset \mathbb{R}^d$, we consider the Dirichlet eigenvalue problem as follows:
\begin{equation*}
\left\{
\begin{aligned}
& L u +\lambda u = 0 \quad\quad in\ \Omega,\\
& u|_{\partial\Omega} =0.\\
\end{aligned}
\right.
\end{equation*}
Let $L $ be the Laplacian, and we can order the eigenvalues as
$$0<\lambda_1(\Omega)<\lambda_2(\Omega)\leqslant\ldots\leqslant\lambda_k(\Omega)\leqslant\ldots.$$
Let $V(\Omega)$ be the volume of the domain $\Omega$ and $V_d$ be the volume of the unit ball in $\mathbb{R}^d$. In 1912, Weyl \cite{cite-key} proved that as $k\rightarrow\infty$,
$$\lambda_k(\Omega) \sim C_d(\frac{k}{V(\Omega)})^{\frac{2}{d}},$$
where $C_d = (2\pi)^2V_d^{-\frac{2}{d}}$. For any planar domain $\Omega$ that tiles $\mathbb{R}^2$, P\'{o}lya \cite{G1961ON} provided the lower bound estimate:
$$\lambda_k(\Omega) \geqslant C_d(\frac{k}{V(\Omega)})^{\frac{2}{d}}.$$
His proof is applicable in any dimension $d\geqslant 2$, and he conjectured that this estimate holds for all bounded domains in $\mathbb{R}^d$. To address this conjecture, Lieb \cite{1979The} obtained the result:
$$\lambda_k(\Omega) \geqslant D_d(\frac{k}{V(\Omega)})^{\frac{2}{d}}.$$
Here, the constant $D_d < C_d$ is proportional to $C_d$. Later in 1983, Li-Yau \cite{1983On} derived the sharp estimate known as the Berezin-Li-Yau inequality:
$$\frac{1}{k}\sum_{i=1}^k \lambda_i(\Omega) \geqslant \frac{d}{d+2}C_d(\frac{k}{V(\Omega)})^{\frac{2}{d}}.$$
 The upper bound estimate was proved by Kr\"{o}ger \cite{1994Estimates} for large $k$ in 1994, which includes an extra term depending on the geometry of $\Omega$.

For the fractional Laplacian $(-\Delta)^s$, $0<s< 1$, we denote by $\lambda_{s,i}(\Omega)$ the $i$-th eigenvalue of $(-\Delta)^s$. The Berezin-Li-Yau type inequality was generalized in \cite{2013ESTIMATES} as:
$$\frac{1}{k}\sum_{i=1}^k \lambda_{s,i}(\Omega) \geqslant \frac{d}{d+2s}C_{d}^{s}(\frac{k}{V(\Omega)})^{\frac{2s}{d}}.$$

In recent years, there has been increasing attention on the analysis of graphs. The graph Laplacian has been extensively studied in the literature, e.g. \cite{1990The,2000Graph,2013Harmonic,BL21}. Some results also exist for the fractional graph Laplacian, such as  \cite{2022Optimal,2022Spectral,Benzi2020,21Decay}. However, in these works, the fractional Laplacian is typically defined on the entire graph, without reference to the Dirichlet eigenvalues of this operator. In this paper, for a locally finite, simple, undirected, connected graph $G$, we first introduce the definition of the fractional Laplacian on a subgraph with Dirichlet boundary condition, and study the Dirichlet eigenvalues.

We recall the setting of graphs. A locally finite, simple, undirected and connected graph $G=(V,E)$ consists of the set of vertices $V$ and the set of edges $E$.  We write $x\in G$ instead of $x\in V$ and denote by $|G|$ the number of vertices in $G$.
Two vertices $x, y$ are called neighbors, denoted by $x\sim y$, if there exists $e \in E$ connecting $x$ and $y$. The graph distance on $G$ is defined by
$$d(x,y):=\inf\{n|x=z_0\sim\dots\sim z_n=y\},$$ 
and we define balls $B(x,r)=\{y\in G: d(x,y)\leqslant r\}$ on $G$. We denote by $V(x,r)=V(B(x,r))$ the number of vertices in $B(x,r)$.
A subgraph $G_1=(V_1, E_1)$ of $G$ means $V_1\subset V$ and $E_1\subset E$. In this paper, we always consider locally finite, simple, undirected and connected graphs.

Recall that the graph Laplacian is defined as, for any $f:G \rightarrow \mathbb{C}$,
$$\Delta f(x) = \sum_{y\sim x}(f(y)-f(x)).$$
Let $(-\Delta)^{\frac{\alpha}{2}}$ be the fractional Laplacian operator, $0<\alpha< 2$. Here we use the equivalent definition (Bochner's definition) in \cite{2015Ten} to describe the fractional Laplacian on $G$:
$$(-\Delta)^{\frac{\alpha}{2}}u(x) = \frac{1}{|\Gamma(-\frac{\alpha}{2})|}\int^{\infty}_0 (u(x)-e^{t\Delta}u(x))t^{-1-\frac{\alpha}{2}}dt,$$
where $u:G \rightarrow \mathbb{C}$, and $e^{t\Delta}$ denotes the heat semigroup, see Definition~\ref{def:b}.

Given a finite subgraph $G_1\subsetneq G$, we introduce the fractional Laplacian operator with Dirichlet boundary condition, denoted by $L^\alpha_{G_1}$, which is defined as:
$$L^\alpha_{G_1} u(x) = (-\Delta)^{\frac{\alpha}{2}} u^{*}(x)|_{G_1}.$$
Here $u:G_1 \rightarrow \mathbb{C}$, and $u^*:G \rightarrow \mathbb{C}$ is the zero extension of $u$. One readily sees that $L^\alpha_{G_1}$ is a real symmetric operator on $G_1$. The Dirichlet eigenvalue problem is given by
$$ L^\alpha_{G_1} u(x) =\lambda u(x).$$
There are $|G_1|$ solutions and corresponding real eigenvalues:
$$0<\lambda_1(G_1)< \lambda_2(G_1)\leqslant\ldots\leqslant\lambda_{|G_1|}(G_1).$$
See Section~\ref{2} for more properties about this operator. It is important to note that on the graph $G$, the eigenvalues of $(-\Delta)^{\frac{\alpha}{2}}$ are actually $\lambda^{\frac{\alpha}{2}}$,  where $\lambda$ is an eigenvalue of $-\Delta$ on $G$. However, the Dirichlet eigenvalues of the fractional Laplacian and the graph Laplacian don't follow this relationship. Thus, it is meaningful to study the estimates for the Dirichlet eigenvalues of the fractional Laplacian.
%We denote by $\mathbb{Z}^d$ the standard integer lattice graph in $\mathbb{R}^d$. On $\mathbb{Z}^d$, the Laplace operator $\Delta $ is defined as, for any  $f:\mathbb{Z}^d \rightarrow \mathbb{C}$,
%$$\Delta f(x) = \sum_{y\sim x}(f(y)-f(x)).$$

%Let $(-\Delta)^{\frac{\alpha}{2}}$ be the fractional Laplacian, $0<\alpha< 2$. Here we use the equivalent definition(Bochner's definition) in \cite{2015Ten} to describe the fractional Laplacian on $\mathbb{Z}^d$:
%$$-(-\Delta)^{\frac{\alpha}{2}}u(x) = \frac{1}{|\Gamma(-\frac{\alpha}{2})|}\int^{\infty}_0 (e^{t\Delta}u(x)-u(x))t^{-1-\frac{\alpha}{2}}dt,$$
%where $u:\mathbb{Z}^d \rightarrow \mathbb{C}$, and $e^{t\Delta}$ denotes the convolution operator with the heat kernel on $\mathbb{Z}^d$, see Definition~\ref{def:b}.

Consider the infinite d-dimensional integer lattice $\mathbb{Z}^d$ as a graph. We study the fractional Laplacian on a finite subgraph $\Omega$ of $\mathbb{Z}^d$ using the Fourier transform. In 2021, Bauer and Lippner \cite{BL21} studied the graph Laplacian eigenvalues with Dirichlet boundary condition on $\Omega$. For the average of the first k Dirichlet eigenvalues, they derived the upper estimate in the spirit of Kr\"{o}ger \cite{1994Estimates} and the lower estimate in the spirit of Li-Yau \cite{1983On}. In this paper, we follow the proof strategies in \cite{1994Estimates,1983On,BL21} to derive the upper and lower bounds for the Dirichlet eigenvalues of the fractional Laplacian.

For the fractional Laplacian on $\mathbb{Z}^d$, we prove that it is well-defined if $u$ satisfies
$$\sum_{y\neq x}\frac{|u(x)-u(y)|}{d(x,y)^{d+\alpha}}<+\infty$$ 
for any vertex $x\in \mathbb{Z}^d$. This result is similar to the continuous case. For a finite subgraph $\Omega \subset \mathbb{Z}^d$, we define a boundary term
%recall that $$\overline{\Omega} = \Omega \cup \{x\in \mathbb{Z}^d| \exists y\in \Omega :y\sim x\},\quad \quad\partial \Omega = \overline{\Omega}\setminus \Omega.$$
$$|\partial^\alpha \Omega| = \sum_{\substack{x\in \Omega^c\\ y\in\Omega}}Q_\alpha (x,y) < +\infty.$$
This could be regarded as the size of the boundary associated with $L^\alpha_{\Omega}$, where 
$$Q_\alpha (x,y) = \frac{1}{|\Gamma(-\frac{\alpha}{2})|}\int^{\infty}_0 t^{-1-\frac{\alpha}{2}}p(t,x,y)dt,$$ 
and $p(t,x,y)$ is the heat kernel on $\mathbb{Z}^d$. Moreover, we prove the order estimates of $Q_\alpha(x,y)$, see Definition~\ref{def:b} and Theorem~\ref{fra}. The following are the main results of this paper.

\tm\label{thm:main1} Let $\Omega$ be a finite subgraph of $\mathbb{Z}^d$, and the operator $L^\alpha_{\Omega}$ be the fractional Laplacian $(-\Delta)^{\frac{\alpha}{2}}$ with Dirichlet boundary condition, $0<\alpha< 2$. Then

$(a):$  for all $1\leqslant k\leqslant\min\{1,\frac{V_d}{2^d}\}|\Omega|$, we have the upper bound estimate
$$ \frac{1}{k}\sum_{i=1}^k \lambda_i(\Omega) \leqslant (2\pi)^\alpha\frac{d}{d+\alpha}(\frac{k}{V_d|\Omega|})^{\frac{\alpha}{d}}+\frac{|\partial^\alpha \Omega|}{|\Omega|}.$$
For $1\leqslant k\leqslant\min\{1,\frac{V_d}{2^{d+1}}\}|\Omega|$, there holds
$$ \lambda_{k+1}(\Omega) \leqslant (2\pi)^\alpha\frac{d\cdot2^{\frac{d+\alpha}{d}}}{d+\alpha}(\frac{k}{V_d|\Omega|})^{\frac{\alpha}{d}}+2\frac{|\partial^\alpha \Omega|}{|\Omega|}.$$

$(b):$ For all $1\leqslant k\leqslant\min\{1,(\frac{2^{1-\frac{1}{\alpha}}\sqrt{3}}{2\pi})^d V_d\}|\Omega|$, the lower bound estimate is given as
 \begin{align*}
 \lambda_k(\Omega)\geqslant \frac{1}{k}\sum_{i=1}^k \lambda_i(\Omega)
 \geqslant (2\pi)^\alpha\frac{d}{d+\alpha}(\frac{k}{V_d|\Omega|})^{\frac{\alpha}{d}}-(2\pi)^{2\alpha}(\frac{1}{12})^{\frac{\alpha}{2}}\frac{d}{d+2\alpha}(\frac{k}{V_d|\Omega|})^
 {\frac{2\alpha}{d}}.
\end{align*}
\tmd

The proof strategy is as follows. The upper bound estimate follows from the Rayleigh-Ritz formula. Let $\{\phi_j\}_{1\leqslant j\leqslant |\Omega|}$ be the corresponding eigenfunctions, $V_k$ be the subspace of $\mathbb{C}^\Omega$ spanned by $\phi_1,\ldots,\phi_k$ and $P_k$ be the orthogonal projection to $V_k$. First by the Rayleigh quotient, we write $\lambda_{k+1}= \inf_{0\neq w\perp V_k}\frac{\langle w,L^\alpha_{\Omega}w\rangle _{\Omega}}{\langle w,w\rangle _{\Omega}}$. For any $g:\Omega\rightarrow\mathbb{C}$, set $g= P_k g+g-P_k g$ and we have Lemma~\ref{lm4}. With the Fourier transform, we choose $g(x)=h_z(x) = e^{i\langle x,z\rangle}$ and integrate in a measurable subset $B$ of $[-\pi,\pi]^d$ to prove Lemma~\ref{lm5}. Finally, we prove $(a)$ by choosing proper set $B$. 

For the lower bound estimate, the key is to prove an integral inequality (Lemma~\ref{lm6}) which is an adaption of Li and Yau's method. To prove this lemma, we construct a radially symmetric function $\varphi(z):[-\pi,\pi]^d\rightarrow\mathbb{C}$ which satisfies $0\leqslant \varphi(z)\leqslant (\Phi(z))^{\frac{\alpha}{2}}$ and
$$\int_{[-\pi,\pi]^d}(\Phi(z))^{\frac{\alpha}{2}}F(z)dz \geqslant \int_{[-\pi,\pi]^d}\varphi(z)F(z)dz.$$
One easily sees that $\widetilde{F}=M 1_{B_R}$ (see Lemma~\ref{lm6}) minimizes the integral $\int_{[-\pi,\pi]^d}\varphi(z)F(z)dz$, which gives the lower bound of $\int_{[-\pi,\pi]^d}(\Phi(z))^{\frac{\alpha}{2}}F(z)dz$. By choosing $F(z) = \sum_{j=1}^k|\langle \phi_j, h_z\rangle_\Omega|^2$, $(b)$ is derived from this lemma.
%The proof of (a) uses the methods in \cite{BL21}, while the proof of (b) follows the strategy in \cite{1983On}. From the properties of Section~\ref{2}, we know the operator $L^\alpha_{\Omega}$ is self-adjoint and positive semi-definite. With the help of a general lemma introduced in \cite{BL21}, we first prove an inequality, Lemma~\ref{lm5}, connecting eigenvalue $\lambda_{k+1}$ and $\sum_{i=1}^k \lambda_i$. Through adjusting the coefficients in this inequality, we derive the upper bound (a). To prove (b), we prove a modification of Lemma 1 in \cite{1983On}, Lemma~\ref{lm6}, which is a integral inequality about function $F$. By choosing $F$ related to the eigenfunctions, Fourier transform in Section~\ref{3} tells us that the integral happens to be $(2\pi)^d\sum_{i=1}^k \lambda_i$, and we finally get the lower bound (b).

Note that on $\Omega$, compared to the Laplacian, the fractional Laplacian is a global operator, whereas the Laplacian is a local operator. Recall that
$$\overline{\Omega} = \Omega \cup \{x\in \mathbb{Z}^d |  \exists y\in \Omega :y\sim x\},\quad\partial \Omega = \overline{\Omega}\setminus \Omega.$$
Here we define a new boundary term $|\partial^\alpha \Omega|$, which depends on all vertices outside $\Omega$, different from the finite sum term $|\partial \Omega|$. In addition, the leading term of our estimates is
$$(2\pi)^\alpha\frac{d}{d+\alpha}(\frac{k}{V_d|\Omega|})^{\frac{\alpha}{d}},$$
which is the same as \cite{2013ESTIMATES}. Letting $\alpha\rightarrow2$, we know the operator tends to the Laplacian with Dirichlet boundary condition. The leading term is consistent with \cite{1983On,1994Estimates}, and our results actually are generalizations of \cite{BL21}.

The paper is organized as follows:
In next section, we introduce the fractional Laplacian on a locally finite connected graph $G$, including a well-definition on $\mathbb{Z}^d$, and define the Dirichlet eigenvalue problem on a finite subgraph. In Section~\ref{sec:pro}, we prove Theorem~\ref{thm:main1}.
%See Figure~\ref{fig1} and Figure~\ref{fig2} for some harmonic functions on penny graphs. Various Liouville type theorems for harmonic functions were proved via the volume doubling property and the Poincar\'e inequality. Following the strategy in \cite{HJLcrelle15}, we study geometric and analytic properties of infinite planar graphs in this paper.

%Penny graphs are the contact graphs of unit circles [1, 2] ¡ª they are formed from
%non-overlapping sets of unit circles by creating a vertex for each circle and an edge
%for each tangency between two circles ¡ª and as such, fit into a long line of graph
%drawing research on contact graphs of geometric objects [3
%¨C7].
\textbf{Acknowledgements.} The author would like to thank Ruowei Li, Fengwen Han, Feng Zhou, Zuoqin Wang for helpful discussions, and thank B. Hua for useful guidance and suggestions. J.X. Wang is supported by Shanghai Science and Technology Program [Project No. 22JC1400100].

\section{Preliminaries}\label{0}

\subsection{The fractional Laplacian on graphs}\label{1} \

\

Let $G=(V,E)$ be a locally finite, simple, undirected and connected graph. The Laplacian on $G$ is defined as
$$\Delta u(x) = \sum_{y\sim x}(u(y)-u(x)),$$
where $u\in C(G):=\{f: V(G)\rightarrow \mathbb{C}\}$, and clearly $\Delta$ is a local operator. We know the heat equation on $G$ is given as
\begin{equation}
\left\{
\begin{aligned}
& \frac{\partial}{\partial t}v(t,x) = \Delta v(t,x) \quad\quad in \ (0,+\infty) \times G, \\
& v(0,x) = u(x)  \quad\quad on \ G. \\
\end{aligned}
\right.
\end{equation}
We denote the solution $v(t,x)$ of (1) by $e^{t\Delta}u(x)$, and the heat kernel by $p(t,x,y)$. We have
$$e^{t\Delta}u(x) = \sum_{y\in G} p(t,x,y)u(y).$$
Here we recall some fundamental properties as follows.

\begin{prop}[\cite{2017Volume,2008Heat}]\label{pro:a}
For $t>0$, and any $x, y\in G$, we have
\begin{enumerate}[(a)]
\item $p(t,x,y)=p(t,y,x)$.
\item $p(t,x,y)>0$.
\item $\sum_{y\in G}p(t,x,y)\leqslant1$. The graph $G$ is stochastically complete if \  $\sum_{y\in G}p(t,x,y)=1$.
\item $\partial_t p(t,x,y)= \Delta_x p(t,x,y)=\Delta_y p(t,x,y)$.
\item $\lim_{t\rightarrow 0}p(t,x,y)=1_{\{x= y\}}$.
\end{enumerate} %We call $\psi$ the quasi-isometry map.
\end{prop}

We always consider stochastically complete graphs. With these properties, we provide the definition of the fractional Laplacian.

\begin{definition}\label{def:b}
For a graph $G = (V,E)$, $0<\alpha< 2$,  the fractional Laplacian is defined as:
\begin{equation}
\begin{aligned}
(-\Delta)^{\frac{\alpha}{2}}u(x) = \frac{1}{|\Gamma(-\frac{\alpha}{2})|}\int^{\infty}_0 (u(x)-e^{t\Delta}u(x))t^{-1-\frac{\alpha}{2}}dt.\\
\end{aligned}
\end{equation}
If the function $u$ satisfies the conditions that
 $\sum_{\substack{y\in G\\ y\neq x}}p(t,x,y)(u(x)-u(y))t^{-1-\frac{\alpha}{2}}$
uniformly converges and
$$|\sum_{\substack{y\in G\\ y\neq x}}p(t,x,y)(u(x)-u(y))t^{-1-\frac{\alpha}{2}}|\leqslant F(t,x)\in L^{1}_t,$$
the fractional Laplacian is well-defined and can be written as:
$$(-\Delta)^{\frac{\alpha}{2}}u(x)= \sum_{\substack{y\in G\\ y\neq x}} Q_\alpha(x,y)(u(x)-u(y)),$$
where for $x\neq y$,
$$0< Q_\alpha(x,y)= \frac{1}{|\Gamma(-\frac{\alpha}{2})|}\int^{\infty}_0 t^{-1-\frac{\alpha}{2}}p(t,x,y)dt <+\infty,$$
and $\sum_{\substack{y\in G\\ y\neq x}}Q_\alpha(x,y)<+\infty$.

\end{definition}

\begin{proof}
	Using Lebesgue's dominated convergence theorem, we have
	\begin{equation*}
		\begin{aligned}
			(-\Delta)^{\frac{\alpha}{2}}u(x) &= \frac{1}{|\Gamma(-\frac{\alpha}{2})|}\int^{\infty}_0 (u(x)-e^{t\Delta}u(x))t^{-1-\frac{\alpha}{2}}dt\\
			&= \frac{1}{|\Gamma(-\frac{\alpha}{2})|}\int^{\infty}_0 (u(x)-\sum_{y\in G}p(t,x,y)u(y))t^{-1-\frac{\alpha}{2}}dt\\
			&= \frac{1}{|\Gamma(-\frac{\alpha}{2})|}\int^{\infty}_0 \sum_{\substack{y\in G\\ y\neq x}}p(t,x,y)(u(x)-u(y))t^{-1-\frac{\alpha}{2}}dt\\
			&= \sum_{\substack{y\in G\\ y\neq x}} Q_\alpha(x,y)(u(x)-u(y)),
		\end{aligned}
	\end{equation*}
	which converges by the conditions.
	
	By Proposition~\ref{pro:a}, it is easy to see $Q_\alpha(x,y)>0$. Using the fact that $0<\sum_{\substack{y\in G\\ y\neq x}}p(t,x,y)\leqslant1$, and
	\begin{equation*}
		\begin{aligned}
	&|\sum_{\substack{y\in G\\ y\neq x}}p(t,x,y)|=|1-p(t,x,x)|=|p(0,x,x)-p(t,x,x)| \\
	= &|t\cdot \partial_t p(\xi,x,x)|=|t\cdot \Delta p(\xi,x,x)|\leqslant ct,
			\end{aligned}
\end{equation*}
	we know $\sum_{\substack{y\in G\\ y\neq x}}Q_\alpha(x,y)<+\infty$. 
\end{proof}

\begin{remark}
	At least, for $u\in l^\infty(G)$, the fractional Laplacian is well-defined.
\end{remark}

On the lattice graph $\mathbb{Z}^d$, we can derive the order estimates of $Q_\alpha(x,y)$. Furthermore, we prove a corollary which gives a suitable domain of definition of the fractional Laplacian. 

\tm\label{fra}
For any $x\neq y$ on $\mathbb{Z}^d$, we have the following order estimates:
$$C_1 d(x,y)^{-d-\alpha}\leqslant Q_\alpha(x,y) \leqslant C_2 d(x,y)^{-d-\alpha},$$
where $C_1,C_2$ are constants only depending on $\alpha$ and the dimension $d$ .
\tmd

To prove this theorem, we need two lemmas which provide the heat kernel estimates. The first lemma in \cite{davies1993large,bauer2017sharp} gives a upper bound estimate for the heat kernel $p(t,x,y)$ for any $t>0$. The second lemma comes from \cite{bauer2015li,horn2019volume}, which proves that $\mathbb{Z}^d$ satisfies the Gaussian heat kernel property for $t\geqslant d(x,y)$, and this result is crucial for us to calculate the order. 

\begin{lemma}[\cite{davies1993large,bauer2017sharp}]\label{heat1}
		
		On $\mathbb{Z}^d$, we have
	$$p(t,x,y)\leqslant e^{-\zeta_1(t,d(x,y))},$$
	where $t>0$ and
	$$\zeta_1(t,r)=r\arcsinh \frac{r}{t}-\sqrt{r^2+t^2}+t.$$
\end{lemma}

\begin{lemma}[\cite{bauer2015li,horn2019volume}]\label{heat2}
	
	The lattice graph $\mathbb{Z}^d$ satisfies the Gaussian heat kernel property if $t\geqslant d(x,y)$ implies
	$$\frac{c_1}{V(x,\sqrt{t})}e^{-c_2\frac{d(x,y)^2}{t}}\leqslant p(t,x,y)\leqslant \frac{c_3}{V(x,\sqrt{t})}e^{-c_4\frac{d(x,y)^2}{t}},$$
	where $c_1,c_2,c_3,c_4$ are constants only depending on the dimension $d$.
\end{lemma}

\begin{proof}[Proof of Theorem~\ref{fra}]

For convenience, we write $D=d(x,y)$ for calculation, and we use uniform constants $C,C'$ only depending on $\alpha$ and the dimension $d$. For 
$$Q_\alpha(x,y)= \frac{1}{|\Gamma(-\frac{\alpha}{2})|}\int^{\infty}_0 t^{-1-\frac{\alpha}{2}}p(t,x,y)dt,$$
we divide the integral into three parts to prove the theorem.

For $0<\frac{t}{D}\leqslant \frac{1}{e}$,
using Lemma~\ref{heat1}, we get the following estimates:
\begin{equation*}
	\begin{aligned}
	p(t,x,y)&\leqslant \exp\left\{-D\arcsinh \frac{D}{t}+\sqrt{D^2+t^2}-t\right\}\\
	&=\exp\left\{-D\ln\left(\frac{D}{t}+\sqrt{1+\left(\frac{D}{t}\right)^2}\right)+\frac{D^2}{t+\sqrt{t^2+D^2}}\right\}\\
	&\leqslant \exp\left\{-D\ln \frac{2D}{t}+D\right\}=e^D\left(\frac{t}{2D}\right)^D.
	\end{aligned}
\end{equation*}
Then we have
\begin{equation*}
	\begin{aligned}
		&\int^{\frac{D}{e}}_0 t^{-1-\frac{\alpha}{2}}p(t,x,y)dt\leqslant \left(\frac{e}{2D}\right)^D\int^{\frac{D}{e}}_0t^{D-1-\frac{\alpha}{2}}dt\\
		=& \left(\frac{e}{2D}\right)^D\cdot \frac{1}{D-\frac{\alpha}{2}}\cdot \left(\frac{D}{e}\right)^{D-\frac{\alpha}{2}}
		\leqslant e^{\frac{\alpha}{2}}\left(\frac{1}{2}\right)^D\frac{1}{D-\frac{\alpha}{2}}\leqslant CD^{-d-\alpha}.
	\end{aligned}
\end{equation*}

For $\frac{1}{e}\leqslant \frac{t}{D}\leqslant 2$, we define a function
$$h(a)=\frac{\ln\left(a+\sqrt{1+a^2}\right)}{a}-\frac{1}{1+\sqrt{1+a^2}}$$
for $a>0$, and it is not difficult to see that $\zeta_1(t,D)=\frac{D^2}{t}h\left(\frac{D}{t}\right)$. Let $b=a+\sqrt{1+a^2}>1$, and $\frac{1}{b}=\sqrt{a^2+1}-a$. We have
\begin{equation*}
	\begin{aligned}
		ah(a)=\ln b-\frac{b-1}{b+1}>0.
	\end{aligned}
\end{equation*}
This implies $h>0$, and 
$$c_5=\inf_{[\frac{1}{2},e]}h(a)>0.$$
Thus we derive that $\zeta_1(t,D)\geqslant c_5\frac{D^2}{t}$, and
\begin{equation*}
	\begin{aligned}
		&\int^{2D}_{\frac{D}{e}} t^{-1-\frac{\alpha}{2}}p(t,x,y)dt
		\leqslant \int^{2D}_{\frac{D}{e}}e^{-c_5\frac{D^2}{t}}t^{-1-\frac{\alpha}{2}}dt\\
		\leqslant& \left(2-\frac{1}{e}\right)De^{-\frac{c_5D}{2}}\left(\frac{D}{e}\right)^{-1-\frac{\alpha}{2}}\leqslant CD^{-d-\alpha}.
	\end{aligned}
\end{equation*}

For $\frac{t}{D}\geqslant 2$, by Lemma~\ref{heat2} we know the Gaussian heat kernel property holds. Since on $\mathbb{Z}^d$, the volume of a ball $B(x,r)$ has the order $r^d$, we derive that
\begin{equation*}
	\begin{aligned}
		&\int_{2D}^{\infty} t^{-1-\frac{\alpha}{2}}p(t,x,y)dt
		\leqslant c_6\int_{2D}^{\infty}e^{-c_4\frac{D^2}{t}}t^{-1-\frac{\alpha}{2}-\frac{d}{2}}dt\\
		\leqslant& c_6 D^{-d-\alpha}\int_0^{\frac{D}{2}}e^{-c_4z}z^{\frac{\alpha}{2}+\frac{d}{2}-1}dz\leqslant CD^{-d-\alpha},
	\end{aligned}
\end{equation*}
where $c_6$ is a constant depending on $c_3$ and the dimension $d$. Similarly, we also have that 
$$\int_{2D}^{\infty} t^{-1-\frac{\alpha}{2}}p(t,x,y)dt\geqslant C'D^{-d-\alpha}.$$

Combining the above calculations, it is not difficult to deduce that the theorem holds.
\end{proof}

\begin{corollary}
The fractional Laplacian on $\mathbb{Z}^d$ is well-defined if $u$ satisfies
$$u\in H_\alpha(\mathbb{Z}^d):=\{u\in C(\mathbb{Z}^d):\ \forall x\in\mathbb{Z}^d,\ \sum_{y\neq x}\frac{|u(x)-u(y)|}{d(x,y)^{d+\alpha}}<+\infty\}.$$
\end{corollary}

\begin{proof}
For convenience, we write $D=d(x,y)$. From $u\in H_\alpha$, we can assume the condition 
$$|u(x)-u(y)|\leqslant C_3D^{d+\alpha}.$$

Let
$$F_k(t,x)=\sum_{D\leqslant k}p(t,x,y)|u(x)-u(y)|t^{-1-\frac{\alpha}{2}}.$$
 We know $0\leqslant F_k\leqslant F_{k+1}$ and 
$$F(t,x)=\sum_{\substack{y\in \mathbb{Z}^d\\ y\neq x}}p(t,x,y)|u(x)-u(y)|t^{-1-\frac{\alpha}{2}}=\lim_{k\rightarrow\infty}F_k(t,x).$$
For any fixed $t>0$, by the proof of Theorem~\ref{fra} we note that
\begin{equation*}
	\begin{aligned}
	F(t,x)&=\sum_{D\leqslant et}p(t,x,y)|u(x)-u(y)|t^{-1-\frac{\alpha}{2}}+\sum_{D> et}p(t,x,y)|u(x)-u(y)|t^{-1-\frac{\alpha}{2}}\\
	&\leqslant C(t,u)+C_3\sum_{D> et}e^D\left(\frac{t}{2D}\right)^DD^{d+\alpha}t^{-1-\frac{\alpha}{2}}\\
	&\leqslant C(t,u)+C_3 \sum_{D> et}\left(\frac{1}{2}\right)^DD^{d+\alpha}t^{-1-\frac{\alpha}{2}}<+\infty.
	\end{aligned}
\end{equation*}
Thus by the monotone convergence theorem, we have
\begin{equation*}
	\begin{aligned}
	&\int_{0}^\infty F(t,x)dt=\lim_{k\rightarrow+\infty}\int_{0}^\infty F_k(t,x)dt\\
	=&\lim_{k\rightarrow+\infty}\sum_{D\leqslant k}Q_\alpha(x,y)|u(x)-u(y)|\\
	\leqslant&C_2\sum_{y\neq x}\frac{|u(x)-u(y)|}{d(x,y)^{d+\alpha}}<+\infty,
	\end{aligned}
\end{equation*}
which implies that $F\in L^1_t$. By Lebesgue's dominated convergence theorem, we prove that the fractional Laplacian is well-defined.
\end{proof}
\subsection{The Dirichlet eigenvalue problem and some properties}\label{2} \

\

Let $G_1\subsetneq G$ be a finite subgraph. By Definition~\ref{def:b}, we know that $(-\Delta)^{\frac{\alpha}{2}}$ is well-defined on $G$. Suppose $u\in C(G_1)$, and we define the Dirichlet eigenvalue problem on $G_1$ as follows:
\begin{equation}
\begin{aligned}
L^\alpha_{G_1} u(x) = (-\Delta)^{\frac{\alpha}{2}} u^{*}(x)|_{G_1}= \lambda u(x).
\end{aligned}
\end{equation}
Here $L^\alpha_{G_1}$ denotes the fractional Laplacian with Dirichlet boundary condition on $G_1$, and $u^{*}(x)\in C(G)$ obtained by extending functions to be 0 on $G_1^c$.

$L^\alpha_{G_1}$ becomes a finite dimensional operator, and thus has $|G_1|$ eigenvalues. Note that for any $ u\in C(G_1)$, we have $u^*\in l^\infty(G)\subset H_\alpha(G)$. Furthermore, the quadratic form $\langle u,L^\alpha_{G_1} u\rangle _{G_1} = \langle u^*,(-\Delta)^{\frac{\alpha}{2}} u^{*}(x)\rangle _{G}$, with the natural Hermitian inner product is well-defined. We give some properties as follows.

\begin{prop}\label{pro:c}
$\forall u,v\in C(G_1)$, $0<\alpha< 2$, we have that
\begin{enumerate}[(a)]
\item $\langle L^\alpha_{G_1} u,v\rangle _{G_1} = \langle u,L^\alpha_{G_1} v\rangle _{G_1}=\frac{1}{2} \sum_{\substack{x,y\in G\\ x\neq y}} Q_\alpha(x,y)(u^*(x)-u^*(y))\overline{(v^*(x)-v^*(y))}=\frac{1}{2} \sum_{\substack{x,y\in G\\ x\neq y}} Q_\alpha(x,y)(u(x)-u(y))\overline{(v(x)-v(y))}+\sum_{\substack{x\in G_1\\ y\in G_1^c}} Q_\alpha(x,y)u(x)\overline{v(x)}$.
\item The eigenvalues are positive real numbers, labeled as
$$0<\lambda_1(G_1)< \lambda_2(G_1)\leqslant\ldots\leqslant\lambda_{|G_1|}(G_1).$$
\item $\phi_1 > 0 $, where $\{\phi_i\}_{1\leqslant i\leqslant |\Omega|}$ denote the corresponding real eigenfunctions.
\item For any $f\in C(G_1)$, the Poisson's equation $L^\alpha_{G_1} u(x)= f$ has a unique solution in $C(G_1)$.
\end{enumerate} %We call $\psi$ the quasi-isometry map.
\end{prop}

\begin{proof}

\

$(a)$: From (3), Proposition~\ref{pro:a} and Definition~\ref{def:b}, we have $Q_\alpha(x,y) = Q_\alpha(y,x)$, and
\begin{equation*}
\begin{aligned}
&\langle L^\alpha_{G_1} u,v\rangle _{G_1} = \langle (-\Delta)^{\frac{\alpha}{2}} u^{*}(x),v^*(x)\rangle _{G}\\
= &\sum_{\substack{x,y\in G\\ x\neq y}} Q_\alpha(x,y)(u^*(x)-u^*(y))\overline{v^*(x)}\\
= &\frac{1}{2} \sum_{\substack{x,y\in G\\ x\neq y}} Q_\alpha(x,y)(u^*(x)-u^*(y))\overline{(v^*(x)-v^*(y))}\\
%= &\sum_{\substack{x\in G_1\\ y\in G_1}} Q_\alpha(x,y)(u(x)-u(y))\overline{v(x)}+\sum_{\substack{x\in G_1\\ y\in G_1^c}} Q_\alpha(x,y)u(x)\overline{v(x)}\\
= &\frac{1}{2} \sum_{\substack{x,y\in G\\ x\neq y}} Q_\alpha(x,y)(u(x)-u(y))\overline{(v(x)-v(y))}+\sum_{\substack{x\in G_1\\ y\in G_1^c}} Q_\alpha(x,y)u(x)\overline{v(x)}\\
= &\langle u,L^\alpha_{G_1} v\rangle _{G_1}.
\end{aligned}
\end{equation*}

$(b)$: The property (a) implies that the operator $L^\alpha_{G_1}$ is self-adjoint and positive definite, which derives that the eigenvalues are positive and real numbers. To prove $\lambda_1$ is single, suppose that there exist real eigenfunctions $\phi_1$ and $\varphi_1$ corresponding to $\lambda_1$, which satisfy $\phi_1 \neq c\varphi_1$. Select a constant $k$ to make $\phi_1 - k\varphi_1$ a sign changing eigenfunction, which contradicts the property (c).

$(c)$: Suppose $\langle \phi_1, \phi_1\rangle _{G_1} = 1$. If $\phi_1$ changes sign, we choose $\phi_1(x_1)<0$ and $\phi_1(x_2)>0$. Define $\widetilde{\phi_1} = |\phi_1|$, and there holds
\begin{equation*}
\begin{aligned}
&\lambda_1 = \langle \phi_1,L^\alpha_{G_1} \phi_1\rangle _{G_1} \\
=& \frac{1}{2} \sum_{\substack{x,y\in G_1\\ x\neq y}} Q_\alpha(x,y)(\phi_1(x)-\phi_1(y))^2+\sum_{\substack{x\in G_1\\ y\in G_1^c}} Q_\alpha(x,y)\phi_1^2(x)\\
>& \frac{1}{2} \sum_{\substack{x,y\in G_1\\ x\neq y}} Q_\alpha(x,y)(\widetilde{\phi_1}(x)-\widetilde{\phi_1}(y))^2+\sum_{\substack{x\in G_1\\ y\in G_1^c}} Q_\alpha(x,y)\widetilde{\phi_1}^2(x)\\
=& \langle \widetilde{\phi_1},L^\alpha_{G_1} \widetilde{\phi_1}\rangle _{G_1} ,
\end{aligned}
\end{equation*}
causing a contradiction.

The above argument shows that $\phi_1\geqslant 0$. Assume $G_1\supset M = \{x\in G_1:\phi_1(x)=0\}\neq\emptyset$. Let
$$a = \min_{x\in G_1}\{\phi_1(x)\neq0\}>0,\
a_1 = \sum_{\substack{x\in M\\ y\in G_1^c}}Q_\alpha(x,y)>0,\
a_2 = \sum_{\substack{x\in M\\ y\in G_1\backslash M}}Q_\alpha(x,y)>0,$$
and choose $0<\epsilon< \frac{2aa_2}{a_1+a_2}<2a$. Then we consider a new function
\begin{equation*}
u(x) =\left\{
\begin{aligned}
& \phi_1(x)\quad ,x\notin M,\\
& \epsilon \quad \quad\quad ,x\in M.\\
\end{aligned}
\right.
\end{equation*}
It is clear that $\langle u, u\rangle _{G_1} = 1+|M|\epsilon^2$, and $\frac{2a-\epsilon}{\epsilon}>\frac{a_1}{a_2}$. By calculation, we have
\begin{equation*}
\begin{aligned}
&\langle u,L^\alpha_{G_1} u\rangle _{G_1} = \frac{1}{2} \sum_{\substack{x,y\in G_1\\ x\neq y}} Q_\alpha(x,y)(u(x)-u(y))^2+\sum_{\substack{x\in G_1\\ y\in G_1^c}} Q_\alpha(x,y)u^2(x)\\
=& \langle \phi_1,L^\alpha_{G_1} \phi_1\rangle_{G_1}+\sum_{\substack{x\in M\\ y\in G_1^c}} Q_\alpha(x,y)\epsilon^2+\sum_{\substack{x\in M\\ y\in G_1\setminus M}} Q_\alpha(x,y)[(\epsilon-u(y))^2-(0-u(y))^2]\\
=& \lambda_1 + a_1\epsilon^2-\sum_{\substack{x\in M\\ y\in G_1\setminus M}} Q_\alpha(x,y)\epsilon(2u(y)-\epsilon)\\
\leqslant& \lambda_1 + a_1\epsilon^2-a_2\epsilon(2a-\epsilon)
= \lambda_1 + (a_1-\frac{2a-\epsilon}{\epsilon}a_2)\epsilon^2 <\lambda_1,
\end{aligned}
\end{equation*}
which is a contradiction.

$(d)$: The unique solution is given by $(L^\alpha_{G_1})^{-1}f$.
\end{proof}

\subsection{The Fourier transform on $\mathbb{Z}^d$}\label{3} \

\

Let $\mathbb{Z}^d$ be the infinite d-dimensional integer lattice, and we choose $S = \{e_1,\ldots,e_d,-e_1,\ldots,-e_d\}$ as the generating set. Here $e_i\in \mathbb{Z}^d$ is the vector whose i-th component is $1$ and the rest are $0$. Thus we can consider $\mathbb{Z}^d$ as the Cayley graph generated by $S$. From Definition~\ref{def:b} and (3), we are clear about the definitions of the fractional Laplacian and the Dirichlet eigenvalue problem on $\mathbb{Z}^d$. In this subsection, we introduce the Fourier transform on $\mathbb{Z}^d$, and derive some critical equalities and inequalities on a finite subgraph $\Omega$ of $\mathbb{Z}^d$.

For $z\in [-\pi,\pi]^d$, let us define $h_z(x):\mathbb{Z}^d\rightarrow \mathbb{C}$ by $ h_z(x)= e^{i\langle x,z\rangle}$, where $\langle x,z\rangle = \sum_{i=1}^d x_iz_i$. Thus, the Fourier transform on $\mathbb{Z}^d$ is defined by
\begin{equation*}
\begin{aligned}
\mathscr{F}:\quad \mathbb{C}^{\mathbb{Z}^d}&\rightarrow \ \mathbb{C}^{[-\pi,\pi]^d}\\
u\ \  &\mapsto \ \ \widehat{u},
\end{aligned}
\end{equation*}
where $u\in l^1(\mathbb{Z}^d)$ and
$$\widehat{u}(z) = \sum_{x\in\mathbb{Z}^d} e^{-i\langle x,z\rangle}u(x)=\langle u,h_z\rangle_{\mathbb{Z}^d},$$
$$u(x) = (2\pi)^{d}\int_{[-\pi,\pi]^d}e^{i\langle x,\xi\rangle}\widehat{u}(\xi)d\xi.$$
Let $\Omega$ be a finite subgraph of $\mathbb{Z}^d$. By (3), if $u\in C(\Omega)$, it is easy to see $u^*\in l^1(\mathbb{Z}^d)$ and $\widehat{L^\alpha_{\Omega} u}$ is well-defined.

To prove Theorem~\ref{thm:main1}, we introduce the function $\Phi (z) = \sum_{i=1}^d(2-2\cos z_i)$ for $z\in[-\pi,\pi]^d$, and prove some relevant lemmas as follows.
\begin{lemma}\label{lm1}
For any $u\in l^1(\mathbb{Z}^d)$, $0<\alpha< 2$, $\xi\in [-\pi,\pi]^d$, we have
$$\widehat{(-\Delta)^{\frac{\alpha}{2}} u}(\xi) = (\Phi (\xi))^\frac{\alpha}{2}\widehat{u}(\xi).$$
\end{lemma}

\begin{proof}
By Definition~\ref{def:b}, we apply the Fourier transform operator to both sides of equation (2) and get
$$\widehat{(-\Delta)^{\frac{\alpha}{2}} u}(\xi) = \frac{1}{|\Gamma(-\frac{\alpha}{2})|}\int^{\infty}_0 (\widehat{u}(\xi)-\widehat{e^{t\Delta}u}(\xi))t^{-1-\frac{\alpha}{2}}dt.$$
For the heat equation (1), we also apply the the Fourier transform operator to both sides, and the equations are transformed into
\begin{equation*}
\left\{
\begin{aligned}
& \frac{d}{dt}\widehat{v}(t,\xi) = -\sum_{j=1}^d(2-e^{i\langle e_j,\xi\rangle}- e^{-i\langle e_j,\xi\rangle})\widehat{v}(t,\xi)\\
&\quad\quad\quad\ \ =-\Phi (\xi)\widehat{v}(t,\xi) \quad\quad in \ (0,+\infty) \times [-\pi,\pi]^d,\\
& \widehat{v}(0,\xi) = \widehat{u}(\xi)  \quad\quad\quad\quad\quad\quad\  in \ [-\pi,\pi]^d. \\
\end{aligned}
\right.
\end{equation*}
The solution is given as $\widehat{e^{t\Delta}u}(\xi)=\widehat{v}(t,\xi)=e^{-\Phi(\xi)t}\widehat{u}(\xi)$. Hence, by the identity
$$\lambda^{\frac{\alpha}{2}}=\frac{1}{|\Gamma(-\frac{\alpha}{2})|}\int^{\infty}_0 (1-e^{-\lambda t})t^{-1-\frac{\alpha}{2}}dt, $$
we derive that
$$\widehat{(-\Delta)^{\frac{\alpha}{2}} u}(\xi) =  (\Phi (\xi))^\frac{\alpha}{2}\widehat{u}(\xi).$$

\end{proof}

With Lemma~\ref{lm1} and the Plancherel formula, we observe that for any $u,v\in C(\Omega)$,
\begin{equation}
\begin{aligned}
&\langle v,L^\alpha_{\Omega} u\rangle _{\Omega}=\langle v^*,(-\Delta)^{\frac{\alpha}{2}} u^*\rangle _{\mathbb{Z}^d}\\
=&\frac{1}{(2\pi)^d}\langle \widehat{v^*},\widehat{(-\Delta)^{\frac{\alpha}{2}} u^*}\rangle _{[-\pi,\pi]^d}\\
=&\frac{1}{(2\pi)^d}\int_{[-\pi,\pi]^d}(\Phi (\xi))^\frac{\alpha}{2}\widehat{v^*}(\xi)\widehat{\overline{u^*}}(\xi)d\xi,
\end{aligned}
\end{equation}
and this integral is bounded.

\begin{lemma}\label{lm2}
Let $\Omega$ be a finite subgraph of $\mathbb{Z}^d$. For any $u\in C(\Omega)$, $0<\alpha< 2$, the following holds:
\begin{enumerate}[(a)]
\item $$\langle u, u\rangle _{\Omega}= \frac{1}{(2\pi)^d}\int_{[-\pi,\pi]^d}|\langle u, h_z\rangle _{\Omega}|^2dz.$$
\item $$\langle u,L^\alpha_{\Omega} u\rangle _{\Omega} = \frac{1}{(2\pi)^d}\int_{[-\pi,\pi]^d}(\Phi (z))^\frac{\alpha}{2}|\langle u, h_z\rangle _{\Omega}|^2dz.$$
\end{enumerate}
\end{lemma}

\begin{proof}

\

$(a)$: The result comes from the Plancherel formula.

$(b)$: From equation (4), let us replace v with u, then we have
\begin{equation*}
\begin{aligned}
\langle u,L^\alpha_{\Omega} u\rangle _{\Omega}
=\frac{1}{(2\pi)^d}\int_{[-\pi,\pi]^d}(\Phi (z))^\frac{\alpha}{2}|\widehat{u^*}(z)|^2dz.
\end{aligned}
\end{equation*}
Since $\widehat{u^*}(z)=\sum_{x\in \mathbb{Z}^d}e^{-i\langle x,z\rangle}u^*(x)=\sum_{x\in \Omega}u(x)\overline{h_z}(x)=\langle u, h_z\rangle _{\Omega}$, the proof is completed.

\end{proof}

\begin{lemma}\label{lm3}
Let $\Omega$ be a finite subgraph of $\mathbb{Z}^d$. For $0<\alpha< 2$, the following holds:
\begin{enumerate}[(a)]
\item $$\langle h_z, h_z\rangle _{\Omega}= |\Omega|.$$
\item $$|\langle h_z,L^\alpha_{\Omega} h_z\rangle _{\Omega}| \leqslant (\Phi (z))^\frac{\alpha}{2}|\Omega|+|\partial^\alpha \Omega|,$$
where $$|\partial^\alpha \Omega| = \sum_{\substack{x\in \Omega^c\\ y\in\Omega}}Q_\alpha (x,y) < +\infty.$$
\end{enumerate}
\end{lemma}

\begin{proof}

\

$(a)$: $$\langle h_z, h_z\rangle _{\Omega}= \sum_{x\in\Omega}h_z(x)\overline{h_z}(x)=\sum_{x\in\Omega}1=|\Omega|.$$

$(b)$: Since
\begin{equation*}
\begin{aligned}
|\langle h_z,L^\alpha_{\Omega} h_z\rangle _{\Omega}| = |\langle h_z^*,(-\Delta)^{\frac{\alpha}{2}} h_z^*\rangle _{\mathbb{Z}^d}|
= |\sum_{x\in\Omega}e^{-i\langle x,z\rangle}\cdot (-\Delta)^{\frac{\alpha}{2}} h_z^*(x)|
\end{aligned}
\end{equation*}
and
\begin{equation*}
\begin{aligned}
&|\langle h_z,(-\Delta)^{\frac{\alpha}{2}} h_z^*\rangle _{\mathbb{Z}^d}|
= |\sum_{x\in\mathbb{Z}^d}e^{-i\langle x,z\rangle}\cdot (-\Delta)^{\frac{\alpha}{2}} h_z^*(x)|\\
= &|(2\pi)^{\frac{d}{2}}\widehat{(-\Delta)^{\frac{\alpha}{2}} h_z^*}(z)| = |(2\pi)^{\frac{d}{2}}(\Phi (z))^\frac{\alpha}{2}\widehat{h_z^*}(z)|\\
= &|(\Phi (z))^\frac{\alpha}{2}\sum_{x\in\Omega}h_z(x)\cdot e^{-i\langle x,z\rangle}|=(\Phi (z))^\frac{\alpha}{2}|\Omega|,
\end{aligned}
\end{equation*}
we can estimate the left-hand side of this inequality,
\begin{equation*}
\begin{aligned}
|\langle h_z,L^\alpha_{\Omega} h_z\rangle _{\Omega}|
\leqslant &(\Phi (z))^\frac{\alpha}{2}|\Omega|+ \sum_{x\in \Omega^c}|e^{-i\langle x,z\rangle}\cdot (-\Delta)^{\frac{\alpha}{2}} h_z^*(x)|\\
\leqslant &(\Phi (z))^\frac{\alpha}{2}|\Omega|+ \sum_{\substack{x\in \Omega^c\\ y\in\mathbb{Z}^d}}|Q_\alpha (x,y)h_z^*(y)|\\
\leqslant &(\Phi (z))^\frac{\alpha}{2}|\Omega|+|\partial^\alpha \Omega|.
\end{aligned}
\end{equation*}
\end{proof}

\begin{remark}\label{rm0}
As $\alpha\rightarrow0$ on $\mathbb{Z}^d$, it is known that  $(-\Delta)^{\frac{\alpha}{2}}$ tends to the identity operator. Indeed, we have
$$(-\Delta)^{s}u(x) = u(x) +s L_\Delta u(x)+o(s)\quad as\ s\rightarrow 0^+,$$
where $L_\Delta$ is the logarithmic Laplacian. By the Fourier transform we define
$$\langle v,L_{\Delta} u\rangle
=\frac{1}{(2\pi)^d}\int_{[-\pi,\pi]^d}\ln\Phi(\xi)\widehat{v}(\xi)\widehat{\overline{u}}(\xi)d\xi,$$
which is well-defined for $u,v\in l^1(\mathbb{Z}^d)$. There have been a lot of researches about the logarithmic Laplacian in the continuous case, and the above properties of the fractional Laplacian may provide insights into the logarithmic Laplacian in the discrete case..
\end{remark}
\section{Proof of Theorem~\ref{thm:main1}}\label{sec:pro}

Let $\Omega$ be a finite subgraph of $\mathbb{Z}^d$. The eigenvalues of the operator $L^\alpha_{\Omega}$ are denoted by
$$0< \lambda_1<\ldots\leqslant\lambda_{|\Omega|},$$
and $\{\phi_j\}_{1\leqslant j\leqslant |\Omega|}$ are the corresponding standardized eigenfunctions. In this section, we follow the methods in \cite{1983On,BL21,1994Estimates} to derive the upper and lower bounds of the Dirichlet eigenvalues of the fractional Laplacian, and prove Theorem~\ref{thm:main1}.

\subsection{Upper bound}

\

\

We begin with a general lemma (proved in \cite{BL21}) on eigenvalues.
\begin{lemma}[Lemma~{3.1} in \cite{BL21}]\label{lm4}
Let L be a self-adjoint, positive semidefinite operator on a finite dimensional, Hermitian, complex vector space W with Hermitian inner product $\langle, \rangle$. Let $0\leqslant \gamma_1\leqslant\ldots\leqslant\gamma_s$ denote its eigenvalues, and let us choose an orthonormal basis of eigenfunctions $f_i: i=1,\ldots,s$, where $f_i$ corresponds to $\gamma_i$. Then for any $1\leqslant k\leqslant s$ and any vector $g\in W$ one has
\begin{equation}
\begin{aligned}
\gamma_{k+1}\langle g,g \rangle \leqslant \langle g,Lg \rangle + \sum_{j=1}^k(\gamma_{k+1}-\gamma_j)|\langle g,f_j \rangle|^2.
\end{aligned}
\end{equation}
\end{lemma}

\

In this lemma, we choose $g=h_z, f_j=\phi_j$, and we can derive an inequality between $\lambda_{k+1}$ and $\sum_{j=1}^k\lambda_j$.
\begin{lemma}\label{lm5}
Let $\Omega$ be a finite subgraph of $\mathbb{Z}^d$. For any $1\leqslant k \leqslant |\Omega|$, $0<\alpha< 2$, and measurable set $B \subset [-\pi,\pi]^d$, we have
\begin{equation*}
\begin{aligned}
\lambda_{k+1}(|\Omega||B|-(2\pi)^d k) \leqslant |\Omega|\cdot\int_{z\in B}(\Phi(z))^{\frac{\alpha}{2}}dz - (2\pi)^d\sum_{j=1}^k\lambda_j
+|B|\cdot |\partial^\alpha \Omega|,
\end{aligned}
\end{equation*}
where $\Phi (z) = \sum_{i=1}^d(2-2\cos z_i)$, and $\lambda_{k+1}=0$ if $k=|\Omega|$.
\end{lemma}

\begin{proof}
In Lemma~\ref{lm4}, we choose
$$W = \mathbb{C}^\Omega, \quad L = L^\alpha_{\Omega}, \quad g = h_z, \quad \gamma_j = \lambda_j, \quad f_j = \phi_j.$$
From Proposition~\ref{pro:c}, we know the above choice satisfies the conditions of Lemma~\ref{lm4}. Integrating both sides of the inequality (5) yields
\begin{equation*}
\begin{aligned}
\lambda_{k+1}\int_{z\in B}\langle h_z,h_z \rangle_\Omega \leqslant \int_{z\in B}\langle h_z,L^\alpha_{\Omega}h_z \rangle_\Omega +
\sum_{j=1}^k(\lambda_{k+1}-\lambda_j)\int_{z\in B}|\langle h_z,\phi_j \rangle_\Omega|^2.
\end{aligned}
\end{equation*}
By Lemma~\ref{lm2} and Lemma~\ref{lm3}, the left-hand side is equal to $\lambda_{k+1}|B||\Omega|$, while
\begin{equation*}
\begin{aligned}
&\quad\quad\quad\quad\quad\quad\quad\quad\quad\ \text{the right-side hand}\\
&\leqslant \int_{z\in B}((\Phi (z))^\frac{\alpha}{2}|\Omega|+|\partial^\alpha \Omega|)dz +
(2\pi)^d\sum_{j=1}^k(\lambda_{k+1}-\lambda_j)\langle \phi_j,\phi_j \rangle_\Omega\\
&= |\Omega|\cdot\int_{z\in B}(\Phi(z))^{\frac{\alpha}{2}}dz + |B|\cdot |\partial^\alpha \Omega|+(2\pi)^d\sum_{j=1}^k(\lambda_{k+1}-\lambda_j),
\end{aligned}
\end{equation*}
which completes the proof of this lemma.
\end{proof}

\

\tm\label{thm2} Let $\Omega$ be a finite subgraph of $\mathbb{Z}^d$. For $1\leqslant k \leqslant \min\{1,\frac{V_d}{2^d}\}|\Omega|$ and $0<\alpha< 2$, we obtain the upper bound estimate:
$$\frac{1}{k}\sum_{j=1}^k\lambda_j \leqslant (2\pi)^\alpha\frac{d}{d+\alpha}(\frac{k}{V_d|\Omega|})^{\frac{\alpha}{d}}+\frac{|\partial^\alpha \Omega|}{|\Omega|}.$$
For $1\leqslant k\leqslant\min\{1,\frac{V_d}{2^{d+1}}\}|\Omega|$, there holds
$$ \lambda_{k+1} \leqslant (2\pi)^\alpha\frac{d\cdot2^{\frac{d+\alpha}{d}}}{d+\alpha}(\frac{k}{V_d|\Omega|})^{\frac{\alpha}{d}}+2\frac{|\partial^\alpha \Omega|}{|\Omega|},$$
where $\lambda_{k+1}=0$ if $k=|\Omega|$.

\tmd

\begin{proof}

\

In Lemma~\ref{lm5}, for $1\leqslant k \leqslant \min\{1,\frac{V_d}{2^d}\}|\Omega|$, we choose the measurable set B as a ball of radius $2\pi(\frac{k}{V_d|\Omega|})^{\frac{1}{d}} $ centered at the origin in $\mathbb{R}^d$, so that the radius $R=2\pi(\frac{k}{V_d|\Omega|})^{\frac{1}{d}}
 \leqslant \pi$ and $|B| = \frac{k(2\pi)^d}{|\Omega|}$. Thus we derive an inequality:
\begin{equation}
\begin{aligned}
 (2\pi)^d \sum_{j=1}^k\lambda_j \leqslant |\Omega|\cdot\int_{z\in B}(\Phi(z))^{\frac{\alpha}{2}}dz +|B|\cdot |\partial^\alpha \Omega|.
\end{aligned}
\end{equation}
In the rest of the proof, we just need to estimate $\int_{z\in B}(\Phi(z))^{\frac{\alpha}{2}}dz$.

We know by calculation,
\begin{equation*}
\begin{aligned}
(\Phi(z))^{\frac{\alpha}{2}}= (\sum_{i=1}^d(2-2\cos z_i))^{\frac{\alpha}{2}}
=(\sum_{i=1}^d4\sin^2 \frac{z_i}{2})^{\frac{\alpha}{2}}
\leqslant (\sum_{i=1}^dz^2_i)^{\frac{\alpha}{2}} = |z|^\alpha.
\end{aligned}
\end{equation*}
This yields that
\begin{equation*}
\begin{aligned}
 \int_{z\in B}(\Phi(z))^{\frac{\alpha}{2}}dz \leqslant \int_{z\in B}|z|^\alpha= dV_d\int^R_0 r^\alpha\cdot r^{d-1}dr = \frac{dV_d}{d+\alpha}R^{d+\alpha}.
\end{aligned}
\end{equation*}
Applying the above inequality into (6), we obtain
\begin{equation*}
\begin{aligned}
\frac{1}{k}\sum_{j=1}^k\lambda_j&\leqslant \frac{dV_d}{d+\alpha}\frac{|\Omega|R^{d+\alpha}}{k(2\pi)^d}+ \frac{|\partial^\alpha \Omega|}{|\Omega|}\\
&=(2\pi)^\alpha\frac{d}{d+\alpha}(\frac{k}{V_d|\Omega|})^{\frac{\alpha}{d}}+\frac{|\partial^\alpha \Omega|}{|\Omega|}.
\end{aligned}
\end{equation*}

Similarly, if $1\leqslant k\leqslant\min\{1,\frac{V_d}{2^{d+1}}\}|\Omega|$, we choose $R'=2^{\frac{1}{d}}R$ , which implies $R' = 2\pi(\frac{2k}{V_d|\Omega|})^{\frac{1}{d}} \leqslant \pi$ and $|B| = \frac{2k(2\pi)^d}{|\Omega|}$. Then we repeat the above proof and ignore the effect of $\sum_{j=1}^k\lambda_j$, obtaining that
\begin{equation*}
\begin{aligned}
k(2\pi)^d\lambda_{k+1}&\leqslant |\Omega|\cdot\int_{z\in B}(\Phi(z))^{\frac{\alpha}{2}}dz +|B|\cdot |\partial^\alpha \Omega|\\
&\leqslant |\Omega|\cdot \frac{dV_d}{d+\alpha}R^{d+\alpha}+ \frac{2k(2\pi)^d|\partial^\alpha \Omega|}{|\Omega|},\\
\end{aligned}
\end{equation*}
which implies the result
$$ \lambda_{k+1} \leqslant (2\pi)^\alpha\frac{d\cdot2^{\frac{d+\alpha}{d}}}{d+\alpha}(\frac{k}{V_d|\Omega|})^{\frac{\alpha}{d}}+2\frac{|\partial^\alpha \Omega|}{|\Omega|}.$$
\end{proof}

\begin{remark}\label{rm1}
Note that as $\alpha\rightarrow2$, the fractional Laplacian tends to the Laplacian, and the upper bound estimate in Theorem~\ref{thm2} also holds, which is the result in \cite{BL21}.
\end{remark}
\

\subsection{Lower bound}

\

\

To prove the lower bound estimate ($(b)$ of Theorem~\ref{thm:main1}), we follow the method by Li and Yau \cite{1983On}. In the following part we prove a modification of
Lemma~{1} from \cite{1983On}, and with Lemma~\ref{lm2} we complete the proof.

\begin{lemma}[Modification of
Lemma~{1} from \cite{1983On}]\label{lm6}
Let $F$ be a real-valued function in $\mathbb{R}^d$, satisfying $0\leqslant F\leqslant M$ and
$$\int_{[-\pi,\pi]^d}F(z)dz\geqslant K.$$
Assume $0<\alpha< 2$, and
$$(\frac{K}{MV_d})^{\frac{1}{d}}\leqslant a=2^{1-\frac{1}{\alpha}}\sqrt{3} <\pi.$$
Then we have the integral inequality
$$\int_{[-\pi,\pi]^d}(\Phi(z))^{\frac{\alpha}{2}}F(z)dz\geqslant \frac{dK}{d+\alpha}(\frac{K}{MV_d})^{\frac{\alpha}{d}}-(\frac{1}{12})^{\frac{\alpha}{2}}\cdot \frac{dK}{d+2\alpha}(\frac{K}{MV_d})^{\frac{2\alpha}{d}},$$
where $\Phi (z) = \sum_{i=1}^d(2-2\cos z_i)$.
\end{lemma}

\begin{proof}

\

Without loss of generality, we assume that $\int_{[-\pi,\pi]^d}F(z)dz= K$. Let us prove the lemma in two steps.

Step 1: At first we are going to construct a radially symmetric function $\varphi(z): [-\pi,\pi]^d\rightarrow \mathbb{R}$ such that
$0 \leqslant \varphi(z)\leqslant(\Phi(z))^{\frac{\alpha}{2}}$. By the Taylor expansion, we observe that if $|z|\leqslant a=2^{1-\frac{1}{\alpha}}\sqrt{3} <\pi$, there holds
\begin{equation*}
\begin{aligned}
&(\Phi(z))^{\frac{\alpha}{2}}= (\sum_{i=1}^d(2-2\cos z_i))^{\frac{\alpha}{2}}\\
\geqslant &(\sum_{i=1}^d(z_i^2-\frac{1}{12}z_i^4))^{\frac{\alpha}{2}}\geqslant ((|z|^2-\frac{1}{12}|z|^4))^{\frac{\alpha}{2}}\\
= &|z|^\alpha(1-\frac{1}{12}|z|^2)^{\frac{\alpha}{2}}\geqslant |z|^\alpha(1-(\frac{1}{12})^{\frac{\alpha}{2}}|z|^\alpha),
\end{aligned}
\end{equation*}
where the last step follows from $(1-x)^{\frac{\alpha}{2}}\geqslant 1-x^{\frac{\alpha}{2}}$ for any $0\leqslant x\leqslant1$. For $|z|\leqslant a $, one easily sees that
$|z|^\alpha(1-(\frac{1}{12})^{\frac{\alpha}{2}}|z|^\alpha)$ is increasing with respect to $|z|$. Thus we define
\begin{equation*}
\varphi(z) = \left\{
\begin{aligned}
& |z|^\alpha - (\frac{1}{12})^{\frac{\alpha}{2}}|z|^{2\alpha} \quad &,\ |z|\leqslant a,\\
& \frac{1}{4}(12)^{\frac{\alpha}{2}} &,\ |z|> a.\\
\end{aligned}
\right.
\end{equation*}
It is clear that $\varphi(z)$ is increasing radially and $0 \leqslant \varphi(z)\leqslant(\Phi(z))^{\frac{\alpha}{2}}$ when $|z|\leqslant a$. If $|z|>a$ and $z\in [-\pi,\pi]^d$, we reduce some $|z_i|$ and choose $|z'|=a<|z|$ which satisfies $|z'_i|\leqslant |z_i|$. Hence, from the expression of $(\Phi(z'))^{\frac{\alpha}{2}}$, we know that
$$(\Phi(z))^{\frac{\alpha}{2}}\geqslant (\Phi(z'))^{\frac{\alpha}{2}}\geqslant \varphi(z')=\frac{1}{4}(12)^{\frac{\alpha}{2}} = \varphi(z).$$

\

Step 2: With the function $\varphi(z)$, we can estimate the lower bound of $\int_{[-\pi,\pi]^d}(\Phi(z))^{\frac{\alpha}{2}}F(z)dz$. Since $F\geqslant 0$, we have
\begin{equation}
\begin{aligned}
\int_{[-\pi,\pi]^d}(\Phi(z))^{\frac{\alpha}{2}}F(z)dz \geqslant \int_{[-\pi,\pi]^d}\varphi(z)F(z)dz.
\end{aligned}
\end{equation}

Choose the constant $R = (\frac{K}{MV_d})^{\frac{1}{d}}\leqslant a=2^{1-\frac{1}{\alpha}}\sqrt{3} <\pi$ to satisfy that $MR^dV_d = K$, and we define
\begin{equation*}
\widetilde{F}(z) = \left\{
\begin{aligned}
& M \ &,\ |z|\leqslant R,\\
& \ 0 &,\ |z|> R.\\
\end{aligned}
\right.
\end{equation*}
Thus,
$$\int_{[-\pi,\pi]^d}\widetilde{F}(z)dz= \int_{[-\pi,\pi]^d}F(z)dz= K.$$
For any function $F$ which satisfies the conditions $0\leqslant F\leqslant M$ and $\int_{[-\pi,\pi]^d}F(z)dz= K$, we claim that $\widetilde{F}$ minimizes the integral
$\int_{[-\pi,\pi]^d}\varphi(z)F(z)dz$ since $\varphi(z)$ is monotonic and radially symmetric. The strict proof is as follows:
\begin{equation*}
\begin{aligned}
&\int_{[-\pi,\pi]^d}\varphi(z)(F(z)-\widetilde{F}(z))dz \\
=& \int_{[-\pi,\pi]^d\setminus B_R}\varphi(z)F(z)dz -\int_{B_R}\varphi(z)(M-F(z))dz\\
\geqslant& \varphi(R)\int_{[-\pi,\pi]^d\setminus B_R}F(z)dz-\varphi(R)\int_{B_R}(M-F(z))dz\\
= &\varphi(R)(\int_{[-\pi,\pi]^d}F(z)dz-\int_{B_R}Mdz) = 0.
\end{aligned}
\end{equation*}
Thus we have the estimate
\begin{equation*}
\begin{aligned}
&\int_{[-\pi,\pi]^d}(\Phi(z))^{\frac{\alpha}{2}}F(z)dz
\geqslant \int_{[-\pi,\pi]^d}\varphi(z)F(z)dz \geqslant M\int_{B_R}\varphi(z)dz\\
=&M \int_{B_R}(|z|^\alpha - (\frac{1}{12})^{\frac{\alpha}{2}}|z|^{2\alpha})dz
= MdV_d \int_0^R(r^\alpha - (\frac{1}{12})^{\frac{\alpha}{2}}r^{2\alpha})r^{d-1}dr\\
=&MdV_d(\frac{R^{d+\alpha}}{d+\alpha}-(\frac{1}{12})^{\frac{\alpha}{2}}\frac{R^{d+2\alpha}}{d+2\alpha})
=\frac{dK}{d+\alpha}(\frac{K}{MV_d})^{\frac{\alpha}{d}}-(\frac{1}{12})^{\frac{\alpha}{2}}\cdot \frac{dK}{d+2\alpha}(\frac{K}{MV_d})^{\frac{2\alpha}{d}}.
\end{aligned}
\end{equation*}
\end{proof}

\tm\label{thm3} Let $\Omega$ be a finite subgraph of $\mathbb{Z}^d$. For $1\leqslant k\leqslant\min\{1,(\frac{2^{1-\frac{1}{\alpha}}\sqrt{3}}{2\pi})^d V_d\}|\Omega|$ and $0<\alpha< 2$, we obtain the lower bound estimate
$$\lambda_k(\Omega)\geqslant \frac{1}{k}\sum_{j=1}^k \lambda_j(\Omega)
 \geqslant (2\pi)^\alpha\frac{d}{d+\alpha}(\frac{k}{V_d|\Omega|})^{\frac{\alpha}{d}}-(2\pi)^{2\alpha}(\frac{1}{12})^{\frac{\alpha}{2}}\frac{d}{d+2\alpha}(\frac{k}{V_d|\Omega|})^
 {\frac{2\alpha}{d}}.$$

\tmd

\begin{proof}

\

The standardized eigenfunctions $\{\phi_j\}_{j=1}^k$ form an orthogonal basis in $l^2(\Omega)$. We define $P_j$ the projection operator on the space spanned by $\phi_j$, and $P$ the projection operator on the space spanned by $\{\phi_j\}_{j=1}^k$. Let $\|\cdot\|$ be the $l^2$ norm on $\Omega$ and $h_z(x) = e^{i\langle z,x \rangle}$ as before. To apply Lemma~\ref{lm6}, let
$$F_j(z) = \langle P_jh_z, P_jh_z\rangle_\Omega = \|P_jh_z\|^2,$$
$$F(z) = \sum_{j=1}^kF_j(z) = \|Ph_z\|^2.$$
By Lemma~\ref{lm2}, we have
$$F(z) = \|Ph_z\|^2  = \sum_{j=1}^k|\langle \phi_j, h_z\rangle_\Omega|^2\leqslant \|h_z\|^2 =|\Omega|, $$
$$\int_{[-\pi,\pi]^d}F(z)dz=\sum_{j=1}^k\int_{[-\pi,\pi]^d}|\langle \phi_j, h_z\rangle_\Omega|^2dz= (2\pi)^d\sum_{j=1}^k\langle \phi_j, \phi_j\rangle_\Omega
=k(2\pi)^d,$$
\begin{equation*}
\begin{aligned}
\int_{[-\pi,\pi]^d}(\Phi(z))^{\frac{\alpha}{2}}F(z)dz &= \sum_{j=1}^k\int_{[-\pi,\pi]^d}(\Phi(z))^{\frac{\alpha}{2}}|\langle \phi_j, h_z\rangle_\Omega|^2dz\\
&=(2\pi)^d\sum_{j=1}^k\langle \phi_j, L^\alpha_{\Omega}\phi_j\rangle_\Omega =(2\pi)^d \sum_{j=1}^k\lambda_j.
\end{aligned}
\end{equation*}
Using Lemma~\ref{lm6} for $M=|\Omega|, K = k(2\pi)^d$, we get the result
$$\lambda_k(\Omega)\geqslant \frac{1}{k}\sum_{j=1}^k \lambda_j(\Omega)
 \geqslant (2\pi)^\alpha\frac{d}{d+\alpha}(\frac{k}{V_d|\Omega|})^{\frac{\alpha}{d}}-(2\pi)^{2\alpha}(\frac{1}{12})^{\frac{\alpha}{2}}\frac{d}{d+2\alpha}(\frac{k}{V_d|\Omega|})^
 {\frac{2\alpha}{d}}.$$
\end{proof}

\begin{remark}
As $\alpha \rightarrow 2$, the fractional Laplacian tends to the Laplacian, and the above lower bound estimate also holds, which is the result proved in \cite{BL21}.
\end{remark}

\textbf{Conflicts of Interests.} The authors declared no potential conflicts of interests with respect to this article.

\textbf{Ethics Approval.}
This study did not involve any human participants or animals, and therefore, ethical approval was not required.

\textbf{Data Availability.}
Data availability is not applicable to this article as no new data were created or analyzed in this study.
\bigskip
\bigskip

\bigskip

\bibliographystyle{alpha}
\bibliography{ckwx}

\end{document}